\documentclass[11pt, letterpaper,reqno]{amsart}

\numberwithin{equation}{section}

\usepackage{amsmath}
\usepackage{amssymb}
\usepackage{amsfonts}
\usepackage{ dsfont }
\usepackage{cases}
\usepackage{cite}
\usepackage[hyphens]{url}
\usepackage{hyperref}
\usepackage{amsthm}
\usepackage{upgreek}
\usepackage{color}
\usepackage{amscd,latexsym,array,hhline}
\usepackage{bm}
\usepackage{enumitem}
\usepackage{graphicx}
\usepackage{placeins}

\providecommand{\norm}[1]{\lVert#1\rVert}
\DeclareMathOperator*{\tr}{tr}
\DeclareMathOperator*{\interior}{int}
\DeclareMathOperator*{\cl}{cl}
\DeclareMathOperator*{\conc}{conc}

\newtheorem{theorem}{Theorem}[section]
\newtheorem{definition}{Definition}[section]
\newtheorem{remark}{Remark}[section]
\newtheorem{assumption}{Assumption}[section]
\newtheorem{proposition}{Proposition}[section]
\newtheorem{lemma}{Lemma}[section]

\title[]{ Martingale optimal transport with stopping} \thanks{We would like to thank Sebastian Hermann, Sigrid K\"{a}llblad and Florian Stebegg for helpful discussions.}
\author[]{Erhan Bayraktar} \thanks{This research was supported in part by the National Science Foundation under grant DMS-1613170.}  
\address{Department of Mathematics, University of Michigan, U.S.A.}
\email{erhan@umich.edu}
\author[]{Alexander M. G. Cox} 
\address{Department of Mathematical Sciences, University of Bath, U.K.}
\email{a.m.g.cox@bath.ac.uk}
\author[]{Yavor Stoev}
\address{Department of Mathematics, University of Michigan, U.S.A.}
\email{ystoev@umich.edu}
\date{\today}
\begin{document}
\keywords{martingale optimal transport, dynamic programming, optimal stopping, stochastic Perron method, viscosity solutions, concave envelope, distribution constraints}
\subjclass[2010]{60G40, 93E20, 91A10, 91A60, 60G07.}
\maketitle

\begin{abstract}
  We solve the martingale optimal transport problem for cost functionals represented
  by optimal stopping problems. The measure-valued martingale approach developed
  in \cite{CoxKal} allows us to obtain an equivalent infinite-dimensional controller-stopper problem.
  We use the stochastic Perron's method and characterize the finite dimensional approximation
  as a viscosity solution to the corresponding HJB equation. It turns out that this
  solution is the concave envelope of the cost function with respect to the atoms
  of the terminal law. We demonstrate the results by
  finding explicit solutions for a class of cost functions.
\end{abstract}

\section{Introduction}
The aim of this article is to solve a class of martingale optimal transport problems
for which the cost functional can be represented as an optimal stopping problem of the
underlying cost function.
Specifically, given a continuous and bounded cost function $f:\nobreak\mathbb{R}\to\mathbb{R}$ we are interested in solving
the martingale optimal transport problem
\begin{align}\label{origprob}
  &\sup_{P_\mu}P^\mathbb{P}(f)\quad\text{with}\quad P^\mathbb{P}(f)=\sup_{\tau\in\mathcal{T}_0} \mathbb{E}[f(M_\tau)].
\end{align}
The outer supremum is taken over $P_\mu$ - the set of all pairs of filtered probability spaces
\break$(\Omega,\mathcal{F},(\mathcal{F}_t)_{t\ge0},\mathbb{P})$ and continuous martingales
$M=(M_t)_{t\ge0}$ on them such that the filtration $(\mathcal{F}_t)_{t\ge0}$ is
generated by a Brownian motion and the terminal law is $M_T\sim\mu$ under $\mathbb{P}$.
The inner stopping problem is over $\mathcal{T}_s$ - the set of all $(\mathcal{F}_{t})$-stopping times taking
values in $[s,T]$ for $s\in[0,T]$ and some fixed terminal time $T>0$.

The duality between martingale optimal transport and robust
pricing problems was studied in a related setting in Dolinsky and Soner \cite{DolSon} for general path-dependent
European-type cost functionals (i.e. payoffs) and continuous models.
Recently Bayraktar and Miller \cite{BayMil} and Beiglb\"ock et al. \cite{Bei1} obtained 
solutions to distribution-constrained optimal stopping problems
by using dynamic programming and martingale transport methods, respectively. In contrast
to our setting, however, the constraints in \cite{BayMil} and \cite{Bei1} are on the
distribution of the stopping times and not on the marginal distribution at the terminal time.
By using the concept of measure-valued martingales Cox and Kallbl\"ad \cite{CoxKal} studied
the robust pricing of Asian-type options subject to a marginal distribution constraint.
The authors cast the original problem into a control theoretic framework and obtained
a viscosity characterization of the solution.

Here we employ the control theoretic approach of \cite{CoxKal} and \cite{BayMil} to
analyze optimal martingale transport problems with cost functionals
which are of American type. The difficulty in our setting is that we have an additional
optimal stopping component. However, the fact that we optimize over continuous models
allows us to prove that the resulting value function is time-independent up to the
terminal time. Since the original problem is infinite dimensional we use the continuity
with respect to the terminal law to restrict it only to measures with finitely many atoms.
Working in a Brownian filtration allows us to recast this finite dimensional approximation
as a recursive sequence of controller-stopper problems with exit-time components.
We prove that the value functions of these problems are viscosity solutions to the
corresponding sequence of elliptic obstacle problems satisfying exact Dirichlet boundary
conditions. We achieve this by applying the stochastic Perron's approach
in the spirit of Bayraktar and Sirbu \cite{BaySir1} where the obstacle problems are
associated with Dynkin games and Rokhlin \cite{Rok} where an elliptic Dirichlet boundary
problem arose from exit-time stochastic control. We circumvent the potential difficulty of
proving a strong comparison result for viscosity sub/supersolutions satisfying \emph{generalized}
boundary conditions (see \cite{Rok}) by using the recursive structure of the problem
to show the exact attainment of these boundary conditions.

The main result in this paper, Theorem~\ref{maintheorem}, is the characterization of the
value function of the finite dimensional martingale transport 
problem as the concave envelope of the pay-off with respect to the probability weights of the
terminal law's atoms. In this final step we use a recent result of Oberman and Ruan \cite{Obe2}
on characterizing convex envelopes as unique viscosity solutions to obstacle
problems with appropriate Dirichlet boundary conditions. One possible application of our
results is the robust pricing
of American options. Indeed, the martingales over which we optimize can be seen as
different models for the stock price with a given marginal distribution at the terminal time.

The rest of the paper is organized as follows: In Section~\ref{sec:pf}, we formulate the
finite dimensional approximation of the Martingale Optimal Transport problem, see
\eqref{vecnotation}. In Section~\ref{sec:main-result}, we employ the stochastic Perron's
method to characterize the value function as the unique viscosity solution of the corresponding
Dirichlet obstacle problem and to show its concave envelope form in an appropriate phase space.
Section~\ref{sec:example} illustrates how our results can be achieved in a probabilistic
framework and provides concrete examples.

\section{Problem formulation}\label{sec:pf}
We define the set of measures $\mathcal{P}$ as
\begin{align*}
  \mathcal{P}:=\{\mu\in\mathcal{B}(\mathbb{R}_+) :\mu(\mathbb{R}_+)=1\text{ and }\int |x| \mu(dx)<\infty \},
\end{align*}
and suppose that the terminal law $\mu$ of the martingales in the optimal transport problem \eqref{origprob}
satisfies $\mu\in\mathcal{P}$. In the usual optimal transport framework we can regard the
probability measures $\mathbb{P}$ contained in $P_\mu$ as transporting the initial
Dirac measure $\delta_{M_0}$ (i.e. the law of $M_0$) to the terminal law $\mu$ under
the cost functional $P^\mathbb{P}$ - both of these laws are known at time $t=0$.
On the other hand, notice that the continuous martingale $M$ satisfies
\begin{align}\label{defS}
  M_t=\mathbb{E}[M_T|\mathcal{F}_t]=\int x\, \xi_{t}(dx) \quad\text{for}\quad t\in[0,T],
\end{align}
where $\xi_t$ is the conditional law of $M_T$ given $\mathcal{F}_t$ under the measure
$\mathbb{P}$. In particular, we have that $\xi_0=\mu$ and $\xi_T=\delta_{M_T}$. Therefore,
similarly to the method proposed in \cite{CoxKal}, we can rewrite \eqref{origprob} in
its measure-valued martingale formulation as
\begin{align}\label{modprob}
  &\sup_{(\xi_t)\in\Xi}\sup_{\tau\in\mathcal{T}_0} \mathbb{E}[f(M_\tau)]\quad\text{subject to}\quad \xi_{0}=\mu,
\end{align}
where $\Xi$ is the set of all terminating measure-valued (i.e. $\mathcal{P}$-valued) martingales (see
Definition 2.7 in \cite{CoxKal}) such that $(\int x\, \xi_t(dx))_{t\ge0}$ is a continuous process a.s. with
respect to the filtered probability space $(\Omega,\mathcal{F},(\mathcal{F}_t)_{t\ge0},\mathbb{P})$ 
for all $(\xi_t)_{t\ge0}\in\Xi$, where $(\mathcal{F}_t)_{t\ge0}$ is a Brownian filtration.
Moreover, as in \cite{CoxKal}, we fix the probability space $(\Omega,\mathcal{F},(\mathcal{F}_t)_{t\ge0},\mathbb{P})$
which does not materially change our conclusions.

Let us write \eqref{modprob} in the Markovian form
\begin{align}\label{markovprob}
  U(t,\xi)=\sup_{(\xi_r)\in\Xi}\sup_{\tau\in\mathcal{T}_t} \mathbb{E}[f(M_\tau)|\xi_t=\xi],
\end{align}
and note that we have the following variant of Lemma 3.1 in \cite{CoxKal} the proof
of which can be found in the appendix:
\begin{lemma}\label{lemmacont}
  If $f$ is non-negative and Lipschitz then the function $U$ is continuous in $\xi$ (in the Wasserstein-1 topology)
  and independent of $t$ for $t\in[0,T)$.
\end{lemma}

The continuity in $\xi$ allows us to apply the finite dimensional reduction from
Section 3.2 in \cite{CoxKal}. In particular, we introduce the set $\mathbb{X}_N=\{x_0,\dots,x_N\}$
where $0\le x_0<x_1<\dots<x_N$ and let $\mathcal{P}^N=\mathcal{P}\cap\mathcal{M}(\mathbb{X}_N)$ 
and $\mathcal{P}(\mathbb{X}_\alpha)=\mathcal{P}\cap\mathcal{M}(\mathbb{X}_\alpha)$ 
for any $\alpha\subseteq \{0,1,\dots,N\}$, where $\mathcal{M}(\mathbb{X}_N)$ resp. $\mathcal{M}(\mathbb{X}_\alpha)$
denote the sets of all measures on $\mathbb{X}_N$ resp. $\mathbb{X}_\alpha:=\{x_i: i\in\alpha\}$.
We assume from now on that the terminal law $\xi$ (i.e. also $\mu$) is an atomic measure and satisfies $\xi\in\mathcal{P}^N$.
Since we work in a Brownian filtration, by martingale representation
for any terminating $\mathcal{P}^N$-valued martingale $(\xi_t)_{t\ge0}$ it is true that
the (nonnegative) martingales $\xi^n_t:=\xi_t(\{x_n\})$ solve an SDE of the form
\begin{align}\label{statesde}
  d\xi^n_t=w^n_td W_t
\end{align}
for $t\ge0$ and $n=0,\dots,N$, where the vector of weights $\mathbf{w}_t=(w_t^0,\dots,w_t^N)$ satisfies $\sum_{n=0}^N w_t^n=0$,
and $\xi^n_t\in\{0,1\}$ implies that $w^n_t=0$. 
The following result, by analogy to Corollary 3.6 in \cite{CoxKal}, follows directly from
Lemma 3.4 in \cite{CoxKal} and allows us to work with a bounded set of controls:
\begin{lemma}\label{prop:tc}
  Under the above assumption that $\mu\in\mathcal{P}^N$, the value function in \eqref{markovprob} for $t\in[0,T)$ reduces to the value function
  \begin{align}\label{atomprob1}
    V(\xi)=\sup_{\mathbf{w}\in\mathcal{A}}\sup_{\tau\in\mathcal{T}_0} \mathbb{E}\left[f\left(\sum_{j=0}^N x_j\, \xi^j_{T^{-1}_{\tau}}\right)|\xi_0=\xi\right],
  \end{align}
  where the admissible control set $\mathcal{A}$ is defined as
  \begin{align*}
    &\mathcal{A}:=\{(\mathbf{w}_r)_{r\ge0}\text{ prog. meas.}: \mathbf{w}_r\in\cl(\mathbb{D}^{N+1})\,,\,\xi^n_r\in\{0,1\}\text{ implies }w^n_r=0\},
  \end{align*}
  with the disk $\mathbb{D}^{k+1}$ being the intersection of the open unit ball with the hyperplane
  $z_1+\dots+z_{k+1}=0$ in $\mathbb{R}^{k+1}$, and $T^{-1}_r$ is the continuous inverse of
  \begin{align}\label{timechange}
    T_r := \int_0^r \lambda_s ds\quad\text{for}\quad r\ge0,
  \end{align}
  where the strictly positive time change rate process $\lambda=(\lambda_r)_{r\ge0}$ satisfies
  \begin{align}  \label{statesde2}
    &\norm{{\mathbf{w}}_r}^2+\lambda_r=1-I_{\{\xi_r=\delta_{x_i}\}}I_{\{T_r=T\}}.
  \end{align}
\end{lemma}
The role of the time change in \eqref{timechange} is to stretch/compress the original
time scale so as to bound the volatility of the state process (i.e. the control process
$\mathbf{w}$). Thus we avoid technical difficulties arising from unbounded control sets
later when proving the viscosity characterization of the value function.

Now notice that the value function $V(\xi)$ can be identified with $\tilde V_N(\xi)$ where for $k=1,\dots,N$, and $\xi\in\mathcal{P}(\mathbb{X}_\alpha)$, with $|\alpha|=k+1$,
we introduce the sequence of problems
\begin{align}\label{recursive}
  \tilde V_k(\xi)=\sup_{\mathbf{w}\in\mathcal{A}^\alpha}\sup_{\tau\in\mathcal{T}_0} \mathbb{E}\Big[&\tilde V_{k-1}(\xi_\sigma)I_{\{T_\sigma\le\tau\}}
  +f\Big(\sum_{j=0}^N x_j\, \xi^j_{T^{-1}_{\tau}}\Big)I_{\{T_\sigma>\tau\}}|\xi_0=\xi\Big],
\end{align}
with
\begin{align}
  \label{Aalphadef}
  \mathcal{A}^\alpha&:=\{(\mathbf{w}_r)_{r\ge0}\text{ prog. meas.}:\mathbf{w}_r\in\cl(\mathbb{D}^{N+1})\,,\\
  &\phantom{:=\{(\mathbf{w}_r)_{r\ge0}\text{ prog. meas.}:\,\,}\,w^i\equiv0\text{ for any } i\in\{0,1,\dots,N\}\setminus\alpha\},\notag\\
  \label{sigmadef}
  \sigma&:=\inf\{s\ge 0 : \xi_s\in\mathcal{P}(\mathbb{X}_{\alpha'})\text{ for some } \alpha'\text{ with } |\alpha'|\le k\text{ or }T_s=T\},
\end{align}
and $\tilde V_0(\xi)=f(x_i)$ for $\xi=\delta_{x_i}$. From now on we will denote the
time changed filtration as $(\mathcal{G}_t)_{t\ge0}:=(\mathcal{F}_{T_t})_{t\ge0}$
and suppress its dependence on $\lambda$ for notational purposes.
The following lemma shows that we can ignore controls which are small enough
and that we can work with stopping times in the time changed filtration.
\begin{lemma}
  The value function $\tilde V_k(\xi)$ can be written as
  \begin{align}\label{recursive1}
    \tilde V_k(\xi)=\sup_{\mathbf{w}\in\interior(\mathcal{A}^\alpha_\varepsilon)}\sup_{\tau\in\mathcal{T}} \mathbb{E}\Big[&\tilde V_{k-1}(\xi_\sigma)I_{\{\sigma\le\tau\}}
    +f\Big(\sum_{j=0}^N x_j\, \xi^j_\tau\Big)I_{\{\sigma>\tau\}}|\xi_0=\xi\Big],\!\!
  \end{align}
  where $\interior(\mathcal{A}_\varepsilon^\alpha):=\{(\mathbf{w}_r)_{r\ge0}\in\mathcal{A}^\alpha: \mathbf{w}_r\in\mathbb{D}^{N+1},\;\xi_r\neq\delta_{x_i}\text{ implies }\norm{\mathbf{w}_r}\ge\varepsilon\}$
  for any $\varepsilon\in[0,1)$ and $\mathcal{T}$ is the set of all $(\mathcal{G}_t)$-stopping times for an appropriately time changed
  filtration $(\mathcal{G}_t)_{t\ge0}$.
\end{lemma}
\begin{proof}
  For any time change rate $\lambda$ we have $\lambda_u>0$
  for $u\ge0$ and from \eqref{statesde2} it follows that $\norm{\mathbf{w}_u}<1$. Moreover,
  since $\lambda$ is strictly positive, we have that $T_r$ and $T^{-1}_t$ are strictly increasing.
  It follows immediately that if $\tau\in[0,T]$ is an $(\mathcal{F}_{t})$-stopping time then $T^{-1}_\tau\ge0$ is
  a $(\mathcal{G}_{t})$-stopping time and, conversely, if $\tau\ge0$ is a $(\mathcal{G}_{t})$-stopping
  time then $T_\tau\in[0,T]$ is an $(\mathcal{F}_{t})$-stopping time. Therefore in
  \eqref{recursive} we can substitute $\mathcal{T}_0$ with $\mathcal{T}$ and $\tau$ with $T_\tau$.

  What is left is to prove that we can take the outer supremum in \eqref{recursive} over $\interior(\mathcal{A}^\alpha_\varepsilon)\subset\interior(\mathcal{A}^\alpha)$.
  For $0<\varepsilon<1$ and any $\mathbf{w}\in\interior(\mathcal{A}^\alpha)\setminus\interior(\mathcal{A}_{\varepsilon}^\alpha)$
  we can choose $\tilde{\mathbf{w}}\in\interior(\mathcal{A}_{\varepsilon}^\alpha)$
  defined as $\tilde{\mathbf{w}}^n_s:=\sqrt{\bar\varepsilon_s}\mathbf{w}^n_{\phi(s)}$
  where
  \begin{align*}
    \phi(s)=\int_0^s\bar\varepsilon_udu\quad\text{with}\quad\bar\varepsilon_s=\frac{\varepsilon^2}{\norm{\mathbf{w}_{\phi(s)}}^2},
  \end{align*}
  and $\phi(s)$ is the right-continuous inverse of the (non-strictly) increasing continuous function $\phi^{-1}(s)$ given by
  \begin{align*}
    \phi^{-1}(s)=\int_0^s\frac{\norm{\mathbf{w}_{u}}^2}{\varepsilon^2}du.
  \end{align*}
  From \eqref{statesde} we see that $\xi^n_r$ (corresponding to the control $\mathbf{w}$) has the same distribution as
  $\tilde\xi^n_{\phi^{-1}(r)}$ (corresponding to the control $\tilde{\mathbf{w}}$).
  Hence, for any $(\mathcal{G}_{t})$-stopping time $\tau$ we have that $\tilde\tau=\phi^{-1}(\tau)$
  is a $(\mathcal{G}_{\phi(t)})$-stopping time such that
  $\xi^n_\tau$ has the same law as $\tilde\xi^n_{\tilde\tau}$.
  We conclude from \eqref{recursive}.
\end{proof}

Before going further we introduce some additional notation. Let $\alpha(\xi)$ be the subset
of elements in $\mathbb{X}_N$ to which the atomic measure $\xi\in\mathcal{P}^N$ prescribes nonzero probability
and notice that we have the consistency conditions
\begin{align*}
  \tilde V_k(\xi)=\tilde V_{|\alpha(\xi)|-1}(\xi)\quad\text{for}\quad k\ge|\alpha(\xi)|.
\end{align*}
For every $\xi\in\mathcal{P}^N$ with $|\alpha(\xi)|=k+1$ it is
true that $\xi=\sum_{j=0}^{k} \xi^{i_j}\delta_{x_{i_j}}$ where $\alpha(\xi)=\{x_{i_0},\dots,x_{i_k}\}\subseteq\mathbb{X}_N$.
Hence, we can identify every $\xi\in\mathcal{P}^N$ with the vector $\xi^\alpha := (\xi^{i_0},\xi^{i_1},\dots,\xi^{i_k})\in\interior(\Delta^{k+1})$
where $\alpha=\{i_0,\dots,i_k\}$ and $\Delta^{k+1}:=\{\mathbf{z}\in\mathbb{R}_{\ge0}^{k+1}: \sum z_i=1\}$. We let
\begin{align}\label{vecnotation}
  V_\alpha(\xi^\alpha)=\tilde V_{|\alpha(\xi)|-1}(\xi),\quad \bar f (\xi^\alpha)=f(\mathbf{x}^\alpha\cdot\xi^\alpha),
\end{align}
where $\mathbf{x}^\alpha:=(x_{i_0},\dots,x_{i_k})$. For any $r\ge0$ and $\mathbf{w}=(w^0,\dots,w^N)\in\interior(\mathcal{A}^\alpha)$
we also let $\bm\xi^{\mathbf{w},r,\xi^\alpha}_u:=(\xi^{i_0,w^{i_0},r}_u,\xi^{i_1,w^{i_1},r}_u,\dots,\xi^{i_k,w^{i_k},r}_u)$,
where $\xi^{i_j,w^{i_j},r}_u$ is the unique strong solution to \eqref{statesde} with control $w^{i_j}$ and initial condition
$\xi^{i_j,w^{i_j},r}_u=\xi^{i_j}$ for $u\le r$. 
Denote by $\xi^{\mathbf{w},r,\xi^\alpha}$ the 
$\mathcal{P}^N$-valued martingale corresponding to $\bm\xi^{\mathbf{w},r,\xi^\alpha}$,
i.e. $\xi^{\mathbf{w},r,\xi^\alpha}_u:=\sum_{j=0}^{k} \xi^{i_j,w^{i_j},r}_u\delta_{x_{i_j}}$.
For short we let $\bm\xi^{\mathbf{w},\xi^\alpha}:=\bm\xi^{\mathbf{w},0,\xi^\alpha}$
and $\xi^{\mathbf{w},\xi^\alpha}:=\xi^{\mathbf{w},0,\xi^\alpha}$.
\section{Viscosity characterization of the value function using stochastic Perron's method}\label{sec:main-result}
We want to obtain the viscosity characterization of the value function $V_\alpha$.
Fix $0<c<1$ and $\alpha\subseteq\{0,\dots,N\}$ with $|\alpha|=k+1\ge2$ for some
integer $k\ge1$. Using \eqref{vecnotation} rewrite the value function from \eqref{recursive1}
as
\begin{align}
  \label{recursivep}
  V_\alpha(\xi^\alpha)&=\sup_{\mathbf{w}\in\interior(\mathcal{A}^\alpha_c)}\sup_{\tau\in\mathcal{T}}\mathbb{E}\left[\tilde V_{k-1}(\xi^{\mathbf{w},r,\xi^\alpha}_{\sigma})I_{\{\sigma\le\tau\}}+\bar f(\bm{\xi}^{\mathbf{w},r,\xi^\alpha}_\tau)I_{\{\sigma>\tau\}} \right],
\end{align}
where $\xi^\alpha\in\Delta^{k+1}$. Our aim is to show that $V_\alpha$ is the unique
viscosity solution (see e.g. Definition 7.4 in \cite{CraIshLio}) to the associated
Dirichlet obstacle problem given by 
\begin{align}
  \label{viscosity}
  \min\Big\{-\sup_{\mathbf{w}\in\mathbb{D}^{k+1}_c}\frac{1}{2}\tr(\mathbf{w}\mathbf{w}'D_{\bm{\xi}}^2V_\alpha), V_\alpha-\bar f\Big\}&=0
  \quad\text{on}\quad\interior(\Delta^{k+1}),\\
  \label{boundarycond}
  \!\!V_\alpha(\xi^\alpha)&=g(\xi^\alpha):= V_{\alpha'}(\xi^{\alpha'})\quad\text{on}\quad \partial\Delta^{k+1},\!
\end{align}
where $\xi^{\alpha'}$ and $\alpha'$ correspond to the nonzero components of $\xi^\alpha$
and $\alpha$, and $\mathbb{D}^{k+1}_c:=\nobreak \{\mathbf{w}\in\mathbb{D}^{k+1}:\norm{\mathbf{w}}>c\}$.
The derivative $D_{\bm{\xi}}^2$ is to be understood in the directional sense - i.e. we restrict
ourselves to second directional derivatives $\tr(\mathbf{w}\mathbf{w}'D_{\bm{\xi}}^2)$ w.r.t. directions lying in the set $\mathbb{D}^{k+1}_c$.

We are now ready to state the main result of the paper - its proof relies on the stochastic
Perron's method and we present it in the next section.
\begin{theorem}\label{maintheorem}
  The function $V_\alpha:\Delta^{k+1}\to\mathbb{R}$ defined in \eqref{recursivep} is
  the unique continuous viscosity solution of the obstacle problem \eqref{viscosity}
  satisfying the Dirichlet boundary condition \eqref{boundarycond}. Moreover, $V_\alpha$
  is the concave envelope of $\bar f$ on $\Delta^{k+1}$ - i.e. denoting the projection
  of $\Delta^{k+1}$ onto $\mathbb{R}_{\ge0}^{k}$ by $\tilde\Delta^{k}$ and the
  projected functions $\tilde V_\alpha,\tilde f:\tilde\Delta^{k}\to\mathbb{R}$ as
  \begin{align}\label{projdef1}
    &\tilde V_\alpha(z_0,\dots,z_{k-1}):=V_\alpha\Big(z_0,\dots,z_{k-1},1-\sum_{i=1}^{k-1} z_i\Big),\\
    \label{projdef2}
    &\tilde f(z_0,\dots,z_{k-1}):=\bar f\Big(z_0,\dots,z_{k-1},1-\sum_{i=1}^{k-1} z_i\Big),
  \end{align}
  the function $\tilde V_\alpha$ is the concave envelope of $\tilde f$.
\end{theorem}

\subsection{Proof of Theorem \ref{maintheorem}}
We begin by introducing the notions of stochastic sub- and supersolutions.
\begin{definition}\label{defstochsubsol}
  The set of stochastic subsolutions to the PDE \eqref{viscosity} with the boundary
  condition \eqref{boundarycond}, denoted by ${\mathcal{V}}^-$, is the set of functions
  $v:\Delta^{k+1}\to\mathbb{R}$ that have the following properties:
  \begin{enumerate}[label=\normalfont(\roman*)]
    \item
      They are continuous and bounded, and satisfy the boundary condition
      \begin{align}\label{subsolcond}
        v(\xi^\alpha)\le g(\xi^\alpha)\quad\text{on}\quad \partial\Delta^{k+1}.
      \end{align}
    \item
      For each $\tau\in\mathcal{T}$ and $\bm{\xi}\in\mathcal{G}_\tau$ with $\mathbb{P}(\bm{\xi}\in\Delta^{k+1})=1$
      there exists a control $\mathbf{w}\in\interior(\mathcal{A}^\alpha)$ such that for any $\rho\in\mathcal{T}$
      with $\rho\in[\tau,\sigma(\tau,\bm\xi,\mathbf{w})]$ we have a.s. that
      \begin{align}\label{submartingale}
        v(\bm{\xi})\le\mathbb{E}[v(\bm\xi^{\mathbf{w},\tau,\bm{\xi}}_{\rho\wedge\tau_*(v)})|\mathcal{G}_\tau],
      \end{align}
      where the $(\mathcal{G}_t)$-stopping times $\sigma(\tau,\bm\xi,\mathbf{w})$
      and $\tau_*(v)$ are defined as
      \begin{align}
        \label{sigmaexitdef}
        &\sigma(\tau,\bm\xi,\mathbf{w}):=\inf\{s\ge \tau : \bm\xi^{\mathbf{w},\tau,\bm\xi}_s\notin\interior(\Delta^{k+1})\},\\
        \label{tauexit}
        &\tau_*(v)\equiv\tau_*(v;\tau,\bm\xi,\mathbf{w}):=\inf\{s\ge \tau :
          v(\bm\xi^{\mathbf{w},\tau,\bm\xi}_s)\le \bar f(\bm\xi^{\mathbf{w},\tau,\bm\xi}_s)\}.
      \end{align}
  \end{enumerate}
\end{definition}
\begin{definition}\label{defstochsupsol}
  The set of stochastic supersolutions to the PDE \eqref{viscosity} with the boundary
  condition \eqref{boundarycond}, denoted by ${\mathcal{V}}^+$, is the set of functions
  $v:\Delta^{k+1}\to\mathbb{R}$ that have the following properties:
  \begin{enumerate}[label=\normalfont(\roman*)]
    \item
      They are continuous and bounded, and satisfy the boundary condition
      \begin{align}\label{supsolcond}
        v(\xi^\alpha)\ge g(\xi^\alpha)\quad\text{on}\quad\partial\Delta^{k+1}.
      \end{align}
    \item
      For each $\tau\in\mathcal{T}$ and $\bm{\xi}\in\mathcal{G}_\tau$ with $\mathbb{P}(\bm{\xi}\in\Delta^{k+1})=1$,
      for any control $\mathbf{w}\in\interior(\mathcal{A}^\alpha_c)$ and any $\rho\in\mathcal{T}$
      with $\rho\in[\tau,\sigma(\tau,\bm\xi,\mathbf{w})]$ we have a.s. that
      \begin{align}\label{supermartingale}
        v(\bm{\xi})\ge\mathbb{E}[v(\bm\xi^{\mathbf{w},\tau,\bm{\xi}}_{\rho})|\mathcal{G}_\tau],
      \end{align}
      where $\sigma(\tau,\bm\xi,\mathbf{w})$ is defined as in \eqref{sigmaexitdef}.
  \end{enumerate}
\end{definition}
Clearly ${\mathcal{V}}^-$ (resp. ${\mathcal{V}}^+$) is nonempty since $\bar f$ is bounded
from below (resp. above) and any constant which is small (large) enough belongs to
${\mathcal{V}}^-$ (resp. ${\mathcal{V}}^+$). Actually, we can easily verify that
$\bar f\in{\mathcal{V}}^-$. The following lemma proves an important property of the
sets ${\mathcal{V}}^-$ and ${\mathcal{V}}^+$.
\begin{lemma}\label{lemmasupinf}
  For any two $v^1,v^2\in{\mathcal{V}}^-$ we have that $v^1\vee v^2\in{\mathcal{V}}^-$.
  For any two $v^1,v^2\in{\mathcal{V}}^+$ we have that $v^1\wedge v^2\in{\mathcal{V}}^+$.
\end{lemma}
\begin{proof}
  We will only prove the first part of the lemma - the second part follows in a similar
  way. Denote $v=v^1\vee v^2$ and notice that item (i) in Definition \ref{defstochsubsol}
  is clearly satisfied by $v$. Now fix $\tau\in\mathcal{T}$ and $\bm{\xi}\in\mathcal{G}_\tau$
  as in item (ii) of Definition \ref{defstochsubsol} and introduce the sequence of
  stopping time, control and state process triples $(\gamma_n,\mathbf{w}^n,\bm{\xi}^n)_{n\ge-1}$
  defined recursively as follows:
  \begin{align*}
    &(\gamma_{-1},\mathbf{w}^{-1},\bm\xi^{-1})\equiv(\gamma_0,\mathbf{w}^0,\bm\xi^0)
    :=(\tau,\mathbf{1}_{\{v^1(\bm\xi)\ge v^2(\bm\xi)\}}\mathbf{w}^{0,1}+\mathbf{1}_{\{v^1(\bm\xi)< v^2(\bm\xi)\}}\mathbf{w}^{0,2},\bm\xi^{{\mathbf{w}}^0,\tau,\bm{\xi}})
  \end{align*}
  where $\mathbf{w}^{0,1},\mathbf{w}^{0,2}$ are the controls corresponding to the
  stochastic subsolutions $v^1, v^2$ starting at the pair $(\tau,\bm{\xi})$, and
  for $n=0, 1, 2\dots$:
  \begin{enumerate}[label=\normalfont(\roman*)]
    \item if $v(\bm\xi^n_{\gamma_{n}})\le \bar f(\bm\xi^n_{\gamma_{n}})$ then
      we set
      \begin{align*}
        (\gamma_{n+1},\mathbf{w}^{n+1},\bm\xi^{n+1}):=(\gamma_{n},\mathbf{w}^{n},\bm\xi^{n}).
      \end{align*}
    \item if $v(\bm\xi^n_{\gamma_{n}})=v^i(\bm\xi^n_{\gamma_{n}})>\bar f(\bm\xi^n_{\gamma_{n}})$
      for $i\in\{1,2\}$ then we set
      \begin{align*}
        &\gamma_{n+1}:=\sigma(\gamma_n,\bm\xi^n_{\gamma_{n}},\mathbf{w}^{n}) \wedge\tau_*(v^i;\gamma_n,\bm\xi^n_{\gamma_{n}},\mathbf{w}^{n})\\
        &\mathbf{w}^{n+1}:=\mathbf{w}^{n+1,i},\quad\bm\xi^{n+1}:=\bm\xi^{\mathbf{w}^{n+1,i},\gamma_{n+1},\bm{\xi}^n_{\gamma_{n+1}}},
      \end{align*}
      where $\mathbf{w}^{n+1,i}$ is the control process corresponding to the stochastic
      subsolution $v^i$ starting at the pair $(\gamma_{n+1},\bm{\xi}^n_{\gamma_{n+1}})$,
      and $\tau_*(v^i;\gamma_n,\bm\xi^n_{\gamma_{n}},\mathbf{w}^{n})$ is defined as
      in \eqref{tauexit}.
  \end{enumerate}
  Define the control $\mathbf{w}$ by
  \begin{align*}
    \mathbf{w}_s:=\sum_{n=1}^\infty\mathbf{1}_{\{s\in[\gamma_n,\gamma_{n+1})\}}\mathbf{w}^n_s
  \end{align*}
  and notice that by construction $\bm\xi^n_s=\bm\xi^{\mathbf{w},\tau,\bm\xi}_s$ for
  $s\in[\gamma_n,\gamma_{n+1}]$ and any $n\ge0$. For any stopping time $\rho\in\nobreak[\tau,\sigma(\tau,\bm\xi,\mathbf{w})]$
  denote $\rho\wedge\gamma_n=\rho_n$. By the definition of the sequence $(\gamma_n,\mathbf{w}^n,\bm{\xi}^n)$
  we get that
  \begin{align*}
    &v(\bm\xi^n_{\rho_{n}})=(\mathbf{1}_{\{v^1\ge v^2\}}v^1+\mathbf{1}_{\{v^1< v^2\}}v^2)(\bm\xi^n_{\rho_{n}})\\
    &\le\mathbb{E}[(\mathbf{1}_{\{v^1(\bm\xi^n_{\rho_{n}})\ge v^2(\bm\xi^n_{\rho_{n}})\}}v^1+\mathbf{1}_{\{v^1(\bm\xi^n_{\rho_{n}})< v^2(\bm\xi^n_{\rho_{n}})\}}v^2)(\bm\xi^{n+1}_{\rho_{n+1}})|\mathcal{G}_{\rho_n}]\\
    &\le\mathbb{E}[v(\bm\xi^{n+1}_{\rho_{n+1}})|\mathcal{G}_{\rho_n}],
  \end{align*}
  and by iterating the above we conclude that
  \begin{align}\label{localized}
    &v(\xi)\le\mathbb{E}[v(\bm\xi^{n+1}_{\rho_{n+1}})|\mathcal{G}_{\tau}]=\mathbb{E}[v(\bm\xi^{\mathbf{w},\tau,\bm\xi}_{\rho_{n+1}})|\mathcal{G}_{\tau}],
  \end{align}
  for any $n\ge0$. Now we apply the same reasoning as in the proof of Lemma 2.3 in
  \cite{BaySir1} to conclude that
  \begin{align*}
    \lim_{n\to\infty} \gamma_n=\sigma(\tau,\bm\xi,\mathbf{w})\wedge\tau_*(v;\tau,\bm\xi,\mathbf{w})\quad a.s.
  \end{align*}
  By taking $n\to\infty$ in \eqref{localized} and using the bounded convergence
  theorem we finally obtain that $v$ satisfies \eqref{submartingale} and, hence, is a stochastic subsolution.
\end{proof}

We introduce the assumption:
\begin{assumption}\label{asscont} The boundary function $g$ is continuous on $\partial\Delta^{k+1}$.
\end{assumption}
\begin{proposition}\label{viscositysupersub}
  Under Assumption \ref{asscont} the lower stochastic envelope $v^-:=\sup_{v\in{\mathcal{V}}^-} v\le V_\alpha$ is
  a viscosity supersolution and the upper stochastic envelope $v^+:=\inf_{v\in{\mathcal{V}}^+} v\ge V_\alpha$ is
  a viscosity subsolution of \eqref{viscosity} and \eqref{boundarycond}.
\end{proposition}
\begin{proof} The proof uses ideas from Theorem 3.1 (and Theorem 4.1) in \cite{BaySir2}
  and Theorem 2 in \cite{Rok}. We repeat the key steps for the lower stochastic envelope
  $v^-$.

  Denote for short $ V\equiv V_\alpha$. It is clear that $v^-\le  V$ since in item (ii)
  of Definition \ref{defstochsubsol} we can choose $\tau=0$, a constant $\bm\xi\in\Delta^{k+1}$
  and $\rho=\sigma(\tau,\bm\xi,\mathbf{w})$ for some control $\mathbf{w}\in\interior(\mathcal{A}^\alpha)$,
  and use the condition \eqref{subsolcond} and \eqref{tauexit}. 

  We will prove the viscosity supersolution property of $v^-$ by contradiction. Take
  a $C^{2}$ test function $\varphi:\Delta^{k+1}\to\mathbb{R}$ such that $v^--\varphi$
  achieves a strict local minimum equal to 0 at some boundary point $\xi_0\in\partial\Delta^{k+1}$
  (the case when $\xi_0\in\interior(\Delta^{k+1})$ is simpler). Assume that $v^-$ is
  not a viscosity supersolution and hence
  \begin{align*}
    \max\big\{(-\sup_{\mathbf{w}\in\mathbb{D}^{k+1}_c}L^\mathbf{w}\varphi)(\xi_0),(\varphi-g)(\xi_0)\big\}<0,
  \end{align*}
  where
  \begin{align*}
    (L^\mathbf{w}\varphi)(\xi):=\frac{1}{2}\tr(\mathbf{w}\mathbf{w}'D_{\bm{\xi}}^2\varphi(\xi)).
  \end{align*}
  It follows that there exists $\tilde{\mathbf{w}}\in\mathbb{D}^{k+1}_c$ such that
  \begin{align}\label{tildew}
    (-L^{\tilde{\mathbf{w}}}\varphi)(\xi_0)<0.
  \end{align}
  By the continuity of $\varphi$, $g$ and the lower semicontinuity of
  $v^-$ we can find a small enough open ball $B(\xi_0,\varepsilon)$ and a small enough
  $\delta>0$ such that
  \begin{align*}
    &(-L^{\tilde{\mathbf{w}}}\varphi)(\xi)<0,\quad \xi\in B(\xi_0,\varepsilon)\cap\Delta^{k+1},\\
    &\varphi<g,\quad \text{on}\quad B(\xi_0,\varepsilon)\cap\partial\Delta^{k+1},\\
    &\varphi(\xi)<v^-(\xi),\quad \xi\in B(\xi_0,\varepsilon)\cap\Delta^{k+1}\setminus\{\xi_0\},\\
    &v^--\delta\ge\varphi\quad\text{on}\quad (\overline{B(\xi_0,\varepsilon)}\setminus B(\xi_0,\varepsilon/2))\cap\Delta^{k+1}.
  \end{align*}
  Using Proposition 4.1 in \cite{BaySir3} together with Lemma \ref{lemmasupinf} above,
  we obtain an increasing sequence of stochastic subsolutions $v_n\in{\mathcal{V}}^-$
  with $v_n\nearrow v^-$. In particular, since $\varphi$ and the $v_n$'s are continuous
  we can use an argument identical to the one in Lemma 2.4 in \cite{BaySir1} to obtain
  for any fixed $\delta'\in(0,\delta)$ a corresponding $v=v_n\in{\mathcal{V}}^-$
  such that
  \begin{align*}
    v-\delta'\ge\varphi\quad\text{on}\quad (\overline{B(\xi_0,\varepsilon)}\setminus B(\xi_0,\varepsilon/2))\cap\Delta^{k+1}.
  \end{align*}
  Now we can choose $\eta\in(0,\delta')$ small enough such that $\varphi^\eta:=\varphi+\eta$
  satisfies
  \begin{align*}
    &(-L^{\tilde{\mathbf{w}}}\varphi^\eta)(\xi)<0,\quad \xi\in B(\xi_0,\varepsilon)\cap\Delta^{k+1},\\
    &\varphi^\eta< g,\quad \text{on}\quad B(\xi_0,\varepsilon)\cap\partial\Delta^{k+1},\\
    &\varphi^\eta<v\quad\text{on}\quad (\overline{B(\xi_0,\varepsilon)}\setminus B(\xi_0,\varepsilon/2))\cap\Delta^{k+1}.
  \end{align*}
  We define
  \begin{align*}
    v^\eta=\begin{cases}
              v\vee\varphi^\eta\;&\text{on}\;\overline{B(\xi_0,\varepsilon)}\cap\Delta^{k+1},\\
              v&\text{otherwise},
           \end{cases}
  \end{align*}
  and notice that $v^\eta$ is continuous and $v^\eta(\xi_0)=v^-(\xi_0)+\eta>v^-(\xi_0)$.
  Since condition \eqref{subsolcond} clearly also holds, we see that $v^\eta$ satisfies
  item (i) of Definition \ref{defstochsubsol}. What is left is to check item (ii) in
  Definition \ref{defstochsubsol} and obtain $v^\eta\in{\mathcal{V}}^-$ which will lead
  to a contradiction since $v^\eta(\xi_0)>v^-(\xi_0)$.

  Choose $\tau\in\mathcal{T}$ and $\bm{\xi}\in\mathcal{G}_\tau$ with $\mathbb{P}(\bm{\xi}\in\Delta^{k+1})=1$,
  and, similarly to the proof of Lemma \ref{lemmasupinf} above, introduce the sequence of
  stopping time, control and state process triples $(\gamma_n,\mathbf{w}^n,\bm{\xi}^n)_{n\ge-1}$
  defined recursively as follows:
  \begin{align*}
    &(\gamma_{-1},\mathbf{w}^{-1},\bm\xi^{-1})\equiv(\gamma_0,\mathbf{w}^0,\bm\xi^0)
    :=(\tau,\tilde{\mathbf{w}}\mathbf{1}_{A}+\bar{\mathbf{w}}^0\mathbf{1}_{A^c},\bm\xi^{{\mathbf{w}}^0,\tau,\bm{\xi}}),
  \end{align*}
  where $\bar{\mathbf{w}}^{0}$ is the control corresponding to the stochastic subsolution
  $v$ starting at the pair $(\tau,\bm{\xi})$, the event $A$ is given by
  \begin{align*}
    A=A(\bm\xi):=\{\bm\xi\in B(\xi_0,\varepsilon/2)\cap\Delta^{k+1}\text{ and }\varphi^\eta(\bm\xi)>v(\bm\xi)\}
  \end{align*}
  and for $n=0, 1, 2\dots$:
  \begin{enumerate}[label=\normalfont(\roman*)]
    \item if $v^\eta(\bm\xi^n_{\gamma_{n}})\le \bar f(\bm\xi^n_{\gamma_{n}})$ then
      we set
      \begin{align*}
        (\gamma_{n+1},\mathbf{w}^{n+1},\bm\xi^{n+1}):=(\gamma_{n},\mathbf{w}^{n},\bm\xi^{n}).
      \end{align*}
    \item if $A(\bm\xi^n_{\gamma_{n}})$ holds then we set
      \begin{align*}
        &\gamma_{n+1}:=\sigma(\gamma_n,\bm\xi^n_{\gamma_{n}},\mathbf{w}^{n}) \wedge\tau_1(\gamma_n,\bm\xi^n_{\gamma_{n}},\mathbf{w}^{n})\wedge\tau_*(\varphi^\eta;\gamma_n,\bm\xi^n_{\gamma_{n}},\mathbf{w}^{n}),\\
        &\mathbf{w}^{n+1}:=\tilde{\mathbf{w}},\quad\bm\xi^{n+1}:=\bm\xi^{\tilde{\mathbf{w}},\gamma_{n+1},\bm{\xi}^n_{\gamma_{n+1}}},
      \end{align*}
      where the $\mathcal{G}_t$-stopping time $\tau_1$ is defined by
      \begin{align*}
        \tau_1(\tau,\bm\xi,\mathbf{w}):=\inf\{s\ge \tau : \bm\xi^{\mathbf{w},\tau,\bm\xi}_s\in\partial B(\xi_0,\varepsilon/2)\},
      \end{align*}
      and $\tau_*$ is defined as in \eqref{tauexit}.
    \item otherwise we set
      \begin{align*}
        &\gamma_{n+1}:=\sigma(\gamma_n,\bm\xi^n_{\gamma_{n}},\mathbf{w}^{n}) \wedge\tau_*(v;\gamma_n,\bm\xi^n_{\gamma_{n}},\mathbf{w}^{n})\\
        &\bm\xi^{n+1}:=\bm\xi^{\mathbf{w}^{n+1},\gamma_{n+1},\bm{\xi}^n_{\gamma_{n+1}}},
      \end{align*}
      where $\mathbf{w}^{n+1}$ is the control process corresponding to the stochastic subsolution $v$
      starting at the pair $(\gamma_{n+1},\bm{\xi}^n_{\gamma_{n+1}})$.
  \end{enumerate}
  By construction we have that $\gamma_n\le\tau_*(v^\eta;\tau,\bm\xi,\mathbf{w})$ where
  the control $\mathbf{w}\in\interior(\mathcal{A}^\alpha_c)$ is defined as
  \begin{align*}
    \mathbf{w}_s :=\sum_{n=1}^{\infty}\mathbf{1}_{\{s\in[\gamma_n,\gamma_{n+1})\}}\mathbf{w}^n_s.
  \end{align*}
  Introduce the event
  \begin{align*}
    B:=\{\gamma_n<\tau_*(v^\eta;\tau,\bm\xi,\mathbf{w})\wedge\sigma(\tau,\bm\xi,\mathbf{w})\quad\text{for all }n\in\mathbb{N}\}
  \end{align*}
  and notice that for each $\omega\in B$ there exists $n_0(\omega)$ such that
  \begin{align}
    \label{ineq1}
    &\varphi^\eta(\bm\xi^{n_0+2l+1}_{\gamma_{n_0+2l+1}})\le \bar f(\bm\xi^{n_0+2l+1}_{\gamma_{n_0+2l+1}})\\
    &\qquad\text{  if } 
    \tau_*(\varphi^\eta;\gamma_{n_0+2l},\bm\xi^{n_0+2l}_{\gamma_{{n_0+2l}}},\mathbf{w}^{n_0+2l})<\tau_1(\gamma_{n_0+2l},\bm\xi^{n_0+2l}_{\gamma_{{n_0+2l}}},\mathbf{w}^{n_0+2l}),\notag\\
    \label{ineq2}
    &v^\eta(\bm\xi^{n_0+2l+1}_{\gamma_{n_0+2l+1}})= v(\bm\xi^{n_0+2l+1}_{\gamma_{n_0+2l+1}})\\
    &\qquad\text{  if } 
    \tau_*(\varphi^\eta;\gamma_{n_0+2l},\bm\xi^{n_0+2l}_{\gamma_{{n_0+2l}}},\mathbf{w}^{n_0+2l})\ge\tau_1(\gamma_{n_0+2l},\bm\xi^{n_0+2l}_{\gamma_{{n_0+2l}}},\mathbf{w}^{n_0+2l}),\notag\\
    \label{ineq3}
    &v(\bm\xi^{n_0+2l+1}_{\gamma_{n_0+2l+1}})\le \bar f(\bm\xi^{n_0+2l+1}_{\gamma_{n_0+2l+1}}),
  \end{align}
  for $l\ge0$. Denoting $\gamma_\infty:=\lim_n\gamma_n$ and noticing that $\bm\xi^{{\mathbf{w}},\tau,\bm{\xi}}_s=\bm\xi^n_{s}$
  for $s\in\nobreak[\gamma_n,\gamma_{n+1})$ we take the limit in \eqref{ineq3} to obtain
  \begin{align}\label{ineq4}
    &v(\bm\xi^{{\mathbf{w}},\tau,\bm{\xi}}_{\gamma_{\infty}})\le \bar f(\bm\xi^{{\mathbf{w}},\tau,\bm{\xi}}_{\gamma_\infty}).
  \end{align}
  Now assume there exists $C\subseteq B$ such that for each $\omega\in C$ we have
  \begin{align*}
    \varphi^\eta(\bm\xi^{{\mathbf{w}},\tau,\bm{\xi}}_{\gamma_{\infty}})> \bar f(\bm\xi^{{\mathbf{w}},\tau,\bm{\xi}}_{\gamma_\infty}),
  \end{align*}
  and conclude from \eqref{ineq1}-\eqref{ineq2} that there exists large enough positive
  integer $M(\omega)$ such that for all $n\ge M$ we have
  \begin{align*}
    v^\eta(\bm\xi^{n}_{\gamma_{n}})= v(\bm\xi^{n}_{\gamma_{n}}).
  \end{align*}
  By taking $n\to\infty$ above we get $v^\eta(\bm\xi^{{\mathbf{w}},\tau,\bm{\xi}}_{\gamma_{\infty}})= v(\bm\xi^{{\mathbf{w}},\tau,\bm{\xi}}_{\gamma_{\infty}})$
  on $C$. Hence, by using \eqref{ineq4} we see that 
  \begin{align*}
    &v^\eta(\bm\xi^{{\mathbf{w}},\tau,\bm{\xi}}_{\gamma_{\infty}})\le \bar f(\bm\xi^{{\mathbf{w}},\tau,\bm{\xi}}_{\gamma_\infty})
  \end{align*}
  on $C$. On the other hand, on $B\setminus C$ we have
  \begin{align*}
    \varphi^\eta(\bm\xi^{{\mathbf{w}},\tau,\bm{\xi}}_{\gamma_{\infty}})\le \bar f(\bm\xi^{{\mathbf{w}},\tau,\bm{\xi}}_{\gamma_\infty})
  \end{align*}
  and again from \eqref{ineq4} we get
  \begin{align*}
    &v^\eta(\bm\xi^{{\mathbf{w}},\tau,\bm{\xi}}_{\gamma_{\infty}})\le \bar f(\bm\xi^{{\mathbf{w}},\tau,\bm{\xi}}_{\gamma_\infty})
  \end{align*}
  on $B\setminus C$. It follows that $\gamma_\infty\ge\tau_*(v^\eta;\tau,\bm\xi,\mathbf{w})$
  on $B$ and from the definition of $B$ we conclude that $\gamma_\infty=\tau_*(v^\eta;\tau,\bm\xi,\mathbf{w})\wedge\sigma(\tau,\bm\xi,\mathbf{w})$.

  Now take any $\rho\in\mathcal{T}$ with $\rho\in[\tau,\sigma(\tau,\bm\xi,\mathbf{w})]$,
  let $\rho\wedge\gamma_n=\rho_n$ and notice that, by It\^o's formula applied to $\varphi^\eta$
  and the subsolution property of $v$, we have
  \begin{align*}
    &v^\eta(\bm\xi^n_{\rho_{n}})=(\mathbf{1}_{A}\varphi^\eta+\mathbf{1}_{A^c}v)(\bm\xi^n_{\rho_{n}})\\
    &\le\mathbb{E}[(\mathbf{1}_{A(\bm\xi^n_{\rho_{n}})}\varphi^\eta+\mathbf{1}_{A(\bm\xi^n_{\rho_{n}})^c}v)(\bm\xi^{n+1}_{\rho_{n+1}})|\mathcal{G}_{\rho_n}]
    \le\mathbb{E}[v^\eta(\bm\xi^{n+1}_{\rho_{n+1}})|\mathcal{G}_{\rho_n}],
  \end{align*}
  and by iterating the above we conclude that
  \begin{align}\label{localized2}
    &v(\bm\xi)\le\mathbb{E}[v(\bm\xi^{n+1}_{\rho_{n+1}})|\mathcal{G}_{\tau}]=\mathbb{E}[\bm\xi^{\mathbf{w},\tau,\bm\xi}_{\rho_{n+1}})|\mathcal{G}_{\tau}].
  \end{align}
  By taking $n\to\infty$ in \eqref{localized2} and using the bounded convergence
  theorem we obtain that $v^\eta$ satisfies item (ii) in Definition \ref{defstochsubsol}
  Hence $v^\eta\in{\mathcal{V}}^-$ and we obtain contradiction and consequently
  the supersolution property of $v^-$.
\end{proof}

\begin{assumption}\label{assconc} 
  The boundary function $g$ is the concave envelope of $\bar f$ on the simplex faces
  $\{\mathbf{z}\in\Delta^{k+1}: z_j=0\}$ for all $j=0,\dots,k+1$.
\end{assumption}
\begin{proposition}\label{viscosityexactbound}
  Under Assumption \ref{assconc} we have that $v^-=v^+=g$ on $\partial\Delta^{k+1}$.
\end{proposition}
\begin{proof}
  Let $\overline v$ be the concave envelope of $\bar f$ on the whole of $\Delta^{k+1}$.
  From Assumption \ref{assconc} it follows that $\overline v=g$ on $\partial\Delta^{k+1}$
  and $\overline v$ satisfies item (i) of Definition \ref{defstochsupsol}.
  Now take any $\tau\in\mathcal{T}$, $\bm{\xi}\in\mathcal{G}_\tau$ with $\mathbb{P}(\bm{\xi}\in\Delta^{k+1})=1$,
  $\mathbf{w}\in\interior(\mathcal{A}^\alpha_c)$ and $\rho\in\mathcal{T}$ with $\rho\in[\tau,\sigma(\tau,\bm\xi,\mathbf{w})]$,
  and notice that, by the It\^o-Tanaka formula (see e.g. Theorem VI.1.5 in \cite{RevYor})
  applied to the concave function $\overline v$ we have
  \begin{align*}
    \mathbb{E}[\overline v(\bm\xi^{\mathbf{w},\tau,\bm{\xi}}_{\rho})|\mathcal{G}_\tau]=\mathbb{E}[\overline v(\bm\xi)
    +\int_\tau^\rho\overline v'(\bm\xi^{\mathbf{w},\tau,\bm{\xi}}_{s})d\bm\xi^{\mathbf{w},\tau,\bm{\xi}}_{s}
    +\int_{\Delta^{k+1}}L_\rho^a\,\overline v''(da)|\mathcal{G}_\tau]\le \overline v(\bm{\xi}),
  \end{align*}
  where $\overline v'$ is the left derivative, the second derivative $\overline v''$ is
  understood in the sense of a negative measure and $L^a$ is the local time at $a$ of the
  process $\bm\xi^{\mathbf{w},\tau,\bm{\xi}}$. Hence, item (ii) of Definition \ref{defstochsupsol}
  is also satisfied and $\overline v$ is a stochastic supersolution. Since $v^+$
  satisfies \eqref{supsolcond} and $v^+\le\overline v$ it follows that $v^+=g$ on
  $\partial\Delta^{k+1}$.

  Fix a constant control $\mathbf{w}\in\interior(\mathcal{A}^\alpha_c)$ and define the function
  $\underline v:\Delta^{k+1}\to\mathbb{R}$ by
  \begin{align}\label{optimalstopping}
    \underline v(\xi^\alpha)&=\sup_{\bar\tau\in\mathcal{T}}\mathbb{E}\left[\tilde V_{k-1}(\xi^{\mathbf{w},\xi^\alpha}_{\sigma})I_{\{\sigma\le\bar\tau\}}+\bar f(\bm{\xi}^{\mathbf{w},\xi^\alpha}_{\bar\tau})I_{\{\sigma>\bar\tau\}} \right].
  \end{align}
  The continuity of $\underline v(\xi^\alpha)$ follows from the boundedness of
  the control $\mathbf{w}$ and standard results on optimal stopping problems (see
  e.g. Theorem 3.1.5 in \cite{Kry}). We have that $\underline v(\xi^\alpha)=V_{\alpha'}(\xi^{\alpha'})=g(\xi^\alpha)$
  for $\xi^\alpha\in\partial\Delta^{k+1}$ and we obtain that item (i) of Definition
  \ref{defstochsubsol} is satisfied. Moreover, the optimal stopping time in \eqref{optimalstopping}
  exists and is equal to $\tau^*=\sigma\wedge\tau_*(\underline v;0,\xi^\alpha,\mathbf{w})$
  and it follows that $\underline v(\bm{\xi}^{\mathbf{w},\xi^\alpha}_{t\wedge\tau^*})$ is a
  martingale (see e.g. Theorems I.2.4 and I.2.7 in \cite{PesShi}). This means that
  \eqref{submartingale} is satisfied with equality and $\underline v$ is a stochastic
  subsolution. By definition we know that $v^-\le g$ on $\partial\Delta^{k+1}$ and $\underline v\le v^-$.
  Hence, we conclude that $v^-=g$ on $\partial\Delta^{k+1}$.
\end{proof}

\begin{proof}[Proof of Theorem \ref{maintheorem}] 
  It is clear that if $|\alpha|=1$ then $ V_\alpha(\xi^\alpha)=\bar f(\xi^\alpha)$ where
  $\xi=\delta_{x_i}$ for some $i$ and $\xi^\alpha=1$. We continue by induction and
  assume that we have proven the statement for all $k'<k$. By the induction hypothesis
  $V_{\alpha'}(\xi^{\alpha'})$ is the concave envelope of $\bar f$ on the corresponding
  to $\alpha'$ simplex face and hence Assumption \ref{assconc} is satisfied. Moreover,
  value functions coincide on the intersection of their corresponding simplex
  faces, and therefore Assumpton \ref{asscont} is also satisfied. Define the Hamiltonian
  $H$ as
  \begin{align*}
    H(A):=-\sup_{\mathbf{w}\in\mathbb{D}^{k+1}_c}\frac{1}{2}\tr(\mathbf{w}\mathbf{w}'A)\quad\text{for}\quad A\in\mathbb{R}^{(k+1)\times (k+1)},
  \end{align*}
  and notice that for small enough $c$ the set $\mathbb{D}^{k+1}_c$ contains all directions in $\mathbb{R}^k$.
  On the other hand, $V_\alpha$ is a viscosity solution to \eqref{viscosity} on $\interior(\Delta^{k+1})$ if
  and only if the projected function $\tilde V_\alpha$ defined in \eqref{projdef1} is a viscosity solution of
  \begin{align}\label{viscosityprojected}
    \min\Big\{-\sup_{\mathbf{w}\in\tilde{\mathbb{D}}^{k}_c}\frac{1}{2}\tr(\mathbf{w}\mathbf{w}'D_{\bm{\xi}}^2\tilde V_\alpha), \tilde V_\alpha-\tilde f\Big\}=0
  \end{align}
  on $\interior(\tilde\Delta^{k})$, where $\tilde{\mathbb{D}}^{k}_c$ is the projection of $\mathbb{D}^{k+1}_c$ onto $\mathbb{R}^{k}$.
  Hence, the function $V_\alpha$ is a viscosity solution to $H(D_{\bm{\xi}}^2V_\alpha)\ge0$ if and only if $\tilde V_\alpha$ is a viscosity
  solution to $-\lambda_k[\tilde V_\alpha]\ge0$, where $\lambda_k[\tilde V_\alpha]$ is the largest eigenvalue of the Hessian $D_{\bm{\xi}}^2\tilde V_\alpha$.
  Therefore we can apply Theorem 1 in \cite{Obe1} to obtain that any continuous viscosity solution
  to \eqref{viscosityprojected} is concave. Moreover, uniqueness of the solution to
  \eqref{viscosityprojected} together with the projected boundary condition
  \begin{align}\label{boundarycondprojected}
    \tilde V_\alpha(\xi^\alpha)&= \tilde V_{\alpha'}(\xi^{\alpha'}),
  \end{align}
  follows from the comparison principle for Dirichlet problems stated in Theorem 2.10 of \cite{Obe2}.
  This leads to uniqueness and comparison principle for our original problem \eqref{viscosity}-\eqref{boundarycond}.
  In particular, by Propositions \ref{viscositysupersub} and \ref{viscosityexactbound} we
  have that $v^+\le v^-$ on $\interior(\Delta^{k+1})$. On the other hand, by Proposition
  \ref{viscositysupersub} we also have $v^-\le V_\alpha\le v^+$ on $\Delta^{k+1}$.
  Therefore, we can conclude that $v^-=V_\alpha=v^+$ on $\Delta^{k+1}$ and $V_\alpha$ is
  the unique viscosity solution of \eqref{viscosity} with the boundary condition
  \eqref{boundarycond}, and the same is true for the projected versions.

  Finally, from Theorem 2 in \cite{Obe1} we have that the concave envelope of the projected cost
  function $\tilde f$ solves \eqref{viscosityprojected}, and since it also clearly
  satisfies \eqref{boundarycondprojected} we conclude from the uniqueness that $\tilde V_\alpha$
  is the concave envelope of $\tilde f$.
\end{proof}

\begin{remark}
  The value function $V_\alpha$ can be regarded as
  the concave envelope on the simplex $\Delta^{k+1}$ of the modified cost function $\bar f$.
  Indeed, we can ignore one direction in the state space
  vector $\bm{\xi}$ due to the fact that $\Delta^{k+1}$ is a k-dimensional simplex and any concave function
  on a k-dimensional simplex in $\mathbb{R}^{k+1}$ is concave in any $k$ of its variables (and vice versa).
  Note that the optimal control weight vector $\mathbf{w}^*$ may not be unique. It is
  determined by the direction on the simplex $\Delta^{k+1}$ for which the second directional
  derivative of the value function $V_\alpha$ is zero - if the value function is linear at a point
  then clearly many directions satisfy this condition.
\end{remark}

\begin{remark}
  When applying the stochastic Perron method to controlled exit time problems one
  needs a comparison result for the corresponding PDE in order to characterize the
  value function as a viscosity solution (see e.g. Definition 2 and Remark 1 in \cite{Rok}).
  These comparison results are of a slightly different nature than the standard ones
  of e.g. Theorems 7.9 and 8.2 in \cite{CraIshLio} - the latter require an apriori
  knowledge of the behaviour of the stochastic semisolutions at the boundary. We were
  able to exploit the specific structure of our exit time problem in Proposition
  \ref{viscosityexactbound} to obtain the behaviour at the boundary of the stochastic
  semisolutions. This allowed the application of the comparison result in \cite{Obe2}.
\end{remark}

\section{Examples}\label{sec:example}
Let us first provide some intuition behind the choice of optimal controls and
stopping times.
We will consider a general class of cost functions - namely all bounded, non-negative
Lipschitz continuous functions $f:\nobreak\mathbb{R}\to\mathbb{R}$. This is the class 
for which Theorem \ref{maintheorem} holds. We will use our concave envelope characterization
to choose the optimal controls and verify that Brownian exit times are optimal.

We abuse notation and regard $\bar f$ as a function on the
projected set of probability vectors $\tilde\Delta^{N}:=\{\mathbf{z}\in\mathbb{R}_{\ge0}^{N}: \sum z_i\le1\}$.
Denote by $\conc(\bar f)$ the concave envelope of $\bar f$ on $\tilde\Delta^{N}$.
For any initial probability vector $z\in\tilde\Delta^{N}$ corresponding to some terminal
law $\mu$, e.g.
\begin{align*}
  \mu=\sum_{i=1}^Nz_i\delta_{x_i}+(1-\sum_{i=1}^Nz_i)\delta_{x_0}, 
\end{align*}
we will find a candidate optimal control weight process $(\mathbf{w}_r)_{r\ge0}$ taking values in
the projected admissible set $\tilde{\mathbb{D}}^{N}_c$ (i.e. the projection of $\mathbb{D}^{N+1}_c$ onto $\mathbb{R}^{N}$)
and a candidate optimal stopping time $\tau_*$ such that the resulting value function will be $\conc(\bar f)$.

The usual characterization of optimal stopping times leads us to choose the candidate
$\tau_*$ as
\begin{align}\label{opttime}
  \tau_*:=\inf\{r\ge0 : \conc(\bar f)(\bm{\xi}^{\mathbf{w},z}_r)=\bar f(\bm{\xi}^{\mathbf{w},z}_r)\}.
\end{align}
In particular, if the initial probability vector $z$ is such that $\conc(\bar f)(z)=\bar f(z)$ we can simply set $\tau_*=\nobreak0$. 
Assume now that $\conc(\bar f)(z)>\bar f(z)$ and note that the point $(z,\conc(\bar f)(z))$
belongs to a planar region of the graph of $\conc(\bar f)(z)$ that contains a point
$(z^{(1)},\conc(\bar f)(z^{(1)}))$ such that $\conc(\bar f)(z^{(1)})=\bar f(z^{(1)})$.
In other words, all points on the line between $(z,\conc(\bar f)(z))$ and $(z^{(1)},\conc(\bar f)(z^{(1)}))$ are also part
of the graph of $\conc(\bar f)$. We choose the control weight process as a constant
vector in the direction of $z-z^{(1)}$, i.e. $\mathbf{w}_r\equiv c_1 (z-z^{(1)})$,
where the constant $c_1$ is such that $\mathbf{w}$ is admissible.
Therefore the probability
vector process $(\bm{\xi}^{\mathbf{w},z}_r)_{r\ge0}$ evolves along the direction $z-z^{(1)}$
and either hits the point $z^{(1)}$ or hits the boundary of $\tilde\Delta^{N}$ at
some point $z^{(2)}$. The point $z^{(2)}$ can be regarded as belonging to a lower
dimensional projected set $\tilde\Delta^{\bar N}:=\{\mathbf{z}\in\mathbb{R}_{\ge0}^{\bar N}: \sum z_i\le1\}$
where $\bar N<N$. If $\conc(\bar f)(z^{(2)})>\bar f(z^{(2)})$, we repeat the same
procedure when choosing a control on this lower dimensional set - clearly this can
happen at most $N$ times.

For simplicity's sake assume that $\conc(\bar f)(z^{(2)})=\bar f(z^{(2)})$. In
other words, by looking at \eqref{statesde} and \eqref{opttime}, we get that $\tau_*$
is the first exit time of a Brownian motion from the interval with endpoints $v_1=\frac{z^{(1)}_0-z_0}{c_1(z_0-z'_0)}$
and $v_2=\frac{z^{(2)}_0-z_0}{c_1(z_0-z'_0)}$.
Using the formula for the Brownian exit times from an interval we obtain that the
projected value function as defined in \eqref{projdef1} satisfies
\begin{align*}
  \tilde V_\alpha(z)=\frac{v_2}{v_2-v_1}\bar f(z^{(1)})+\frac{-v_1}{v_2-v_1}\bar f(z^{(2)})
\end{align*}
and the point $(z,\tilde V_\alpha(z))$ lies on the line going through $(z,\conc(\bar f)(z))$ and $(z',\conc(\bar f)(z'))$,
hence $\tilde V_\alpha(z)=\conc(\bar f)(z)$. Similar calculation is valid for the
case $\conc(\bar f)(z^{(2)})>\bar f(z^{(2)})$.

Finally, by application of the It\^o-Tanaka formula as in the proof of Proposition
\ref{viscosityexactbound} we conclude that $\conc(\bar f)$ bounds the value function
from above, and therefore the two coincide.

\begin{remark}[Generalized Put options]
  In fact, if the cost function is of the form
  \begin{align*}
    f(s)=(g(s))^+,
  \end{align*}
  for some concave function $g$, by direct calculation we can check that the candidate
  control and stopping time described above are optimal among those controls that follow
  a fixed direction and those stopping times that are Brownian exit times from an interval. 
  By applying Theorem \ref{maintheorem} we see that optimization over this class is sufficient.
\end{remark}

In what follows, using the observations above, we will construct the optimal controls and
stopping times explicitly for a piece-wise linear cost function which can be thought of
as a call option spread.

\subsection{Call option spread}
We let $f$ take the form
\begin{align*}
  f(s)=(s-K_1)^+-(s-K_2)^+
\end{align*}
for $K_1\in(-1,1)$, $K_2\in(0,1)$ and $K_1<K_2$, which can be seen as a bull call spread. Set $N=2$, $\mathbb{X}_N=\{-1,0,1\}$ and assume that the law of $M_T$ is
given by
\begin{align*}
  \mu = (1-\gamma-\beta)\delta_{-1}+\beta \delta_0 + \gamma \delta_1,
\end{align*}
for $0<\gamma,\beta<1$ such that $0<\gamma+\beta<1$.
Therefore, the initial probability vector is
\begin{align*}
  {\xi}^\alpha\equiv(\xi^{0}_0,\xi^{1}_0,\xi^{2}_0) = (1-\gamma-\beta,\beta,\gamma)\in\interior(\Delta^3)
\end{align*}
where $\alpha=\{0,1,2\}$. From the definition of the process $M$ in \eqref{defS} it follows that
\begin{align}\label{stateprocessM}
  M_t = \gamma_{T^{-1}_t} - (1-\gamma_{T^{-1}_t}-\beta_{T^{-1}_t})=2\gamma_{T^{-1}_t}+\beta_{T^{-1}_t} - 1\quad\text{for}\quad t\in[0,T],
\end{align}
where $\beta_r = \xi^1_r$ and $\gamma_r = \xi^2_r$ for $r\ge0$. 
We introduce the constants $s^{-101}=2\gamma+\beta - 1$, $s^{01}=\frac{\gamma}{\gamma+\beta}$,
$s^1=1$ and $s^0=0$ corresponding to the value of $M_0$ taking various atoms
of $\mathbb{X}_N$ into account.
We use the notation $V_\alpha(\beta,\gamma):=V_\alpha({\xi}^\alpha)$ and $\bar f(\beta,\gamma):=\bar f({\xi}^\alpha)$.

We will now describe how to obtain a guess for the value function which, as expected, will turn
out to be the concave envelope of the modified cost function $\bar f$.
Notice that $f$ is nondecreasing and achieves its maximum for any $s\ge K_2$ and its minimum
for any $s\le K_1$. Therefore, for the martingale state process $\bm{\xi}^{\mathbf{w},\xi^\alpha}$ (or equivalently the law process $\xi^{\mathbf{w},\xi^\alpha}$),
we want to offset any decrease of probability mass on the interval $(K_2,\infty)$ with
a corresponding decrease on the interval $(-\infty,K_1)$.
We consider the following cases:
\begin{enumerate}
  \item 
    Assume $M_0\equiv s^{-101}\ge K_2$. Then it is optimal to stop immediately, i.e. choose an
    optimal stopping time $\tau_*=0$ and obtain $V_\alpha(\beta,\gamma)=K_2-K_1$.
  \item
    Assume $s^{01}\ge K_2>s^{-101}$ and let the constant $\eta\in[0,1-\gamma-\beta)$ be such that $\frac{\gamma-\eta}{\gamma+\beta+\eta}=K_2$.  
    Then it is optimal to choose
    a stopping time $\tau_*$ and a control process $\mathbf{w}_r\equiv(w^0_r,w^1_r,w^2_r)=(-c_1-\nobreak\frac{\beta}{\gamma}c_1,\frac{\beta}{\gamma}c_1,c_1)$
    for any $r\in[0,\tau_*]$, where the constant $c_1>0$ is such that $\mathbf{w}$ is an admissible control and the optimal stopping
    time $\tau_*$ is the first exit time of $\gamma_r$ from the interval $(0,\frac{\gamma}{\gamma+\beta+\eta})$.
    Note that this choice of $\mathbf{w}$ is not unique.

    Equivalently, by using \eqref{stateprocessM}, we see that $\tau_*$ is the first exit time of $M_{T_r}$ from the interval $(-1,K_2)$.
    This corresponds to letting the law $\xi^{\mathbf{w},\xi^\alpha}$ evolve until
    the stopping time $\tau_*$ when it separates into two measures of the form
    \begin{align*}
      \xi_{\tau_*}^{\mathbf{w},\xi^\alpha}=\begin{cases}
        \frac{\gamma\delta_1+\beta\delta_0+\eta\delta_{-1}}{\gamma+\beta+\eta}&\text{ with probability }\gamma+\beta+\eta,\\
        \delta_{-1}&\text{ with probability }1-(\gamma+\beta+\eta).
      \end{cases}
    \end{align*}
    By the definition of $\eta$ we have that $\gamma+\beta+\eta=\frac{2\gamma+\beta}{K_2+1}$ and therefore $V_\alpha(\beta,\gamma)=\frac{2\gamma+\beta}{K_2+1}(K_2-K_1)$.
  \item
    Assume $K_2>s^{01}$ and let the constant $\eta\in(0,\beta)$ be such that $\frac{\gamma}{\gamma+\eta}=K_2$. 
    Then we choose a stopping time $R_1$ and a control process $\mathbf{w}_r\equiv(w^0_r,w^1_r,w^2_r)=(-c_1-\nobreak\frac{\eta-\beta(\gamma+\eta)}{\gamma-\gamma(\gamma+\eta)}c_1,\frac{\eta-\beta(\gamma+\eta)}{\gamma-\gamma(\gamma+\eta)}c_1,c_1)$
    for any $r\in[0,R_1]$, where the constant $c_1>0$ is such that $\mathbf{w}$ is an admissible control and the stopping
    time $R_1$ is the first exit time of $\gamma_r$ from the interval $(0,\frac{\gamma}{\gamma+\eta})$.
    Equivalently, by using \eqref{stateprocessM}, we see that $R_1$ is the first exit time of $M_{T_r}$ from the interval $\big(-\frac{1-\gamma-\beta}{1-\gamma-\eta},K_2\big)$.
    This corresponds to letting the law $\xi^{\mathbf{w},\xi^\alpha}$ evolve until
    time $R_1$ when it separates into two measures of the form
    \begin{align*}
      \xi_{R_1}^{\mathbf{w},\xi^\alpha}=\begin{cases}
        \frac{\gamma\delta_1+\eta\delta_0}{\gamma+\eta}&\text{ with probability }\gamma+\eta,\\
        \frac{(\beta-\eta)\delta_0+(1-\beta-\gamma)\delta_{-1}}{1-\gamma-\eta}&\text{ with probability }1-(\gamma+\eta).
      \end{cases}
    \end{align*}

    In addition, if $s^0\le K_1$, we choose the optimal stopping time as $\tau_*\equiv R_1$
    and we have $V_\alpha(\beta,\gamma)=\frac{\gamma}{K_2}(K_2-K_1)$. This is due to
    the fact that if $\gamma_{R_1}=0$ (i.e. the atom $\{ 1\}$ dies) it is not worth to evolve the law $\xi^{\mathbf{w},\xi^\alpha}$ further because the
    cost function $f$ will be $0$ under any combination of the atoms $\{0,-1\}$. In other words
    we gain nothing from transferring probability mass between the atoms $0$ and $-1$.

    On the other hand, if we also have that $s^0>K_1$, on the event $A:=\{\gamma_{R_1}=0\}$ we let the control process
    be $\mathbf{w}_r=(-w^1_{R_1},w^1_{R_1},0)$ for $r\in(R_1,R_2]$ and set the optimal stopping time
    \begin{align*}
      \tau_*=R_1\mathbf{1}_{A^c}+R_2\mathbf{1}_{A},
    \end{align*}
    where the stopping time $R_2$ is the first exit time of $\beta_u$ from the interval $(0,1)$ for $u>R_1$.
    Equivalently, by using \eqref{stateprocessM}, we see that $R_2$ is the first exit time of $M_{T_r}$ from
    the interval $\big(-1,0\big)$ for $r> R_1$.
    This corresponds to further evolving the law $\xi^{\mathbf{w},\xi^\alpha}$ until at
    the stopping time $R_2>R_1$ it splits into three measures of the form
    \begin{align*}
      \xi_{R_2}^{\mathbf{w},\xi^\alpha}=\begin{cases}
        \frac{\gamma\delta_1+\eta\delta_0}{\gamma+\eta}&\text{ with probability }\gamma+\eta,\\
        \delta_0&\text{ with probability }\beta-\eta,\\
        \delta_{-1}&\text{ with probability }1-\beta-\gamma.
      \end{cases}
    \end{align*}
    Therefore we have
    \begin{align*}
      V_\alpha(\beta,\gamma)=\frac{\gamma}{K_2}(K_2-K_1)+(\beta-\eta)(-K_1)=\gamma(1-K_1)-\beta K_1.
    \end{align*}
\end{enumerate}
\begin{figure}[!htb]
  \centering
  \scalebox{0.60}{\includegraphics[trim={3cm 0 0 0},clip]{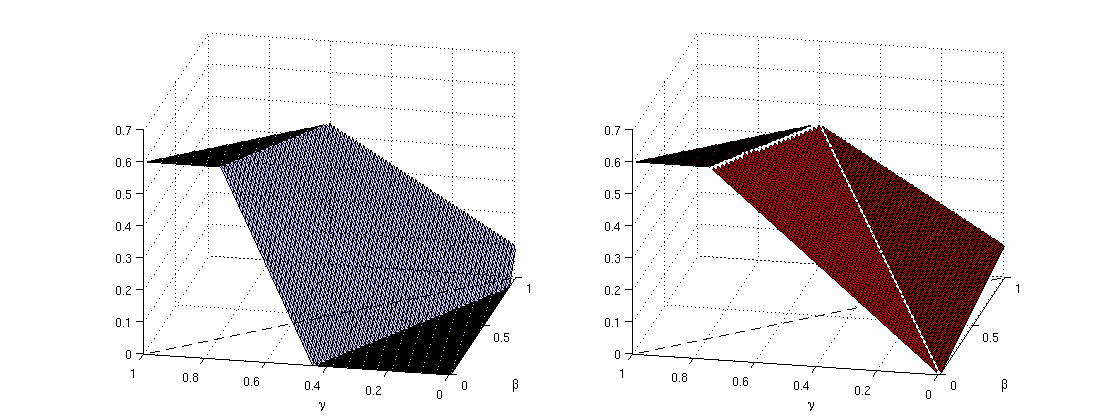}}
  \caption{The modified cost function $\bar f(\beta,\gamma)$ on the left plotted together with the projected value function $V_\alpha(\beta,\gamma)$ on the right for $K_1=\nobreak-0.1$ and $K_2=0.5$.
  The three triangular planar regions correspond to the three cases above. It is evident that $V_\alpha(\beta,\gamma)$ is the concave envelope of $\bar f(\beta,\gamma)$.}\label{fig1}
\end{figure}
\FloatBarrier

The candidate value function $V_\alpha(\beta,\gamma)$ is given by
\begin{align*}
  V_\alpha(\beta,\gamma)=\begin{cases}
    K_2-K_1&\textnormal{(i)}\quad s^{-101}\ge K_2\\
    \frac{2\gamma+\beta}{K_2+1}(K_2-K_1)&\textnormal{(ii)}\quad s^{01}\ge K_2>s^{-101}\\
    \frac{\gamma}{K_2}(K_2-K_1)&\textnormal{(iii)}\quad  K_2>s^{01},s^0\le K_1\\
    \gamma(1-K_1)-\beta K_1&\textnormal{(iv)}\quad  K_2>s^{01},s^0> K_1
  \end{cases}
\end{align*}
and it is the concave envelope of $\bar f(\beta,\gamma)$ (see Figure \ref{fig1}). \footnote{It turns out that the value function in this example is the same as in the Asian option setting of \cite{CoxKal}; see the example in Section 4.2 therein. This is because under their optimal model the stock price is a fixed random variable which is given by the average of our measure valued martingale at $\tau^*$ using \eqref{defS}.}

\appendix
\section{Proof of Lemma \ref{lemmacont}}
\begin{proof}\let\qed\relax
  In order to prove the independence in the $t$ variable we choose $0\le t_1<t_2< T$ and notice that $U(t_1,\xi)\ge U(t_2,\xi)$.
  Indeed, the supremum in \eqref{markovprob} corresponding to $U(t_1,\xi)$ is taken
  over a larger set of stopping times than the one corresponding to $U(t_2,\xi)$.
  Conversely, for any $\xi\in\Xi$ and $\tau\in\mathcal{T}_{t_1}$
  we can choose $\tilde\xi\in\Xi$ and $\tilde\tau\in\tilde{\mathcal{T}_{t_2}}$
  such that 
  \begin{align*}
    \tilde\tau=a\tau+b,\qquad \tilde\xi_{at+b}=\xi_{t}
  \end{align*}
  with $a=\frac{T-t_2}{T-t_1}$ and $b=\frac{T(t_2-t_1)}{T-t_1}$. This choice leads
  to 
  \begin{align*}
    \int x\, \xi_{\tau}(dx)=\int x\, \tilde\xi_{{\tilde \tau}}(dx)
  \end{align*}
  which allows us to conclude that $U(t_2,\xi)\ge U(t_1,\xi)$ and hence
  $U(t_2,\xi)= U(t_1,\xi)$ and we have independence in $t$ for $t\in[0,T)$.
  
    To prove the continuity in $\xi$ we first observe (e.g. see Lemma 3.1 in \cite{CoxKal}) that if $(\xi_r)_{r\ge0}\in\Xi$ with
  $\xi_t=\xi$ and $d_{\mathcal{W}_1}(\xi_t,\xi')<\varepsilon$ (here $d_{\mathcal{W}_1}$ is the Wasserstein-1 metric)
  then there is $(\xi'_r)_{r\ge0}\in\Xi$ with $\xi'_t=\xi'$ such that $\mathbb{E}[|\int x\,\xi_\tau(dx)-\int x\,\xi'_\tau(dx)||\mathcal{F}_t]<\varepsilon$
  for all $\tau\in \mathcal{T}_t$ with some fixed $\lambda\in\Lambda$.
  Indeed, we know that $\xi_s=\mathbb{E}[\xi_T|\mathcal{F}_s]$ 
  and we can define
  \begin{align*}
    \xi'_s(dy)=\mathbb{E}\left[\int \xi_T(dx)m(x,dy)|\mathcal{F}_s\right],\quad s\ge t,
  \end{align*}
  where the Borel family of probability measures $m(x,dy)$ is obtained by the disintegration
  of the transport plan $\Gamma(dx,dy)=\xi_t(dx)m(x,dy)$ such that $\Gamma(\mathbb{R}_+,dy)=\xi'(dy)$,
  $\Gamma(dx,\mathbb{R}_+)=\xi_t(dx)$ and $\int \int |x-y|\Gamma(dx,dy)<\varepsilon$.
  By optional stopping we get
  \begin{align*}
    \left|\int x\,\xi_\tau(dx)-\int x\,\xi'_\tau(dx)\right|\le \mathbb{E}\left[\int\int|x-y|\xi_T(dx)m(x,dy)|\mathcal{F}_\tau\right]
  \end{align*}
  and hence
  \begin{align*}
    \mathbb{E}\left[\left|\int x\,\xi_\tau(dx)-\int x\,\xi'_\tau(dx)\right|\Big|\mathcal{F}_t\right]\le \int \int |x-y|\Gamma(dx,dy)<\varepsilon.
  \end{align*}
  Denote by $M^{\xi}$ the process corresponding to the measure-valued martingale $(\xi_r)_{r\ge0}$
  from \eqref{defS}. By the Lipschitz property of $f$ and the above inequality we get
  \begin{align*}
    \mathbb{E}\left[\left|f\big(M^{\xi'}_\tau\big)-f\big(M^{\xi}_\tau\big)\right||\mathcal{F}_t\right]<\varepsilon.
  \end{align*}
  Now fix $\varepsilon'>0$ and consider $ \xi,\xi'\in\mathcal{P}$ such that
  $d_{\mathcal{W}_1}(\xi,\xi')<\varepsilon'/2$. From the reasoning above, we can
  choose $(\xi_r)_{r\ge0},(\xi'_r)_{r\ge0}\in\Xi$ with $\xi_t=\xi$ and $\xi'_t=\xi'$ such that
  $U(t,\xi)\le\sup_{\tau\in\mathcal{T}_t} \mathbb{E}[f\big(M^{\xi}_\tau\big)|\mathcal{F}_t]+\varepsilon'/2$
  and $\mathbb{E}\left[\left|f\big(M^{\xi'}_\tau\big)-f\left(M^{\xi}_\tau\right)\right||\mathcal{F}_t\right]<\varepsilon'/2$. 
  Therefore we obtain
  \begin{align*}
    U(t,\xi)\le\sup_{\tau\in\mathcal{T}_t} \mathbb{E}[f(M^{\xi}_\tau)|\mathcal{F}_t]+\varepsilon'/2\le\sup_{\tau\in\mathcal{T}_t} \mathbb{E}[f(M^{\xi'}_\tau)|\mathcal{F}_t]+\varepsilon'\le U(t,\xi')+\varepsilon',
  \end{align*}
  and by symmetry we get $|U(t,\xi)-U(t,\xi')|\le \varepsilon'$ and continuity follows.
\end{proof}

\bibliography{modelindependent}
\bibliographystyle{plain}

\end{document}